\documentclass{amsart}

\usepackage{amsmath,amstext,amssymb,mathrsfs,amscd,amsthm}
\usepackage{amsfonts}
\usepackage[all]{xy}
\usepackage{booktabs}
\usepackage{verbatim}
\usepackage{enumerate}
\usepackage[bookmarks=false]{hyperref}

\usepackage{graphicx}                   
\usepackage{subfig}
\usepackage{xspace}

\newtheorem{lemma}{Lemma}[section]
\newtheorem{proposition}[lemma]{Proposition}
\newtheorem{theorem}{Theorem}
\newtheorem{corollary}[lemma]{Corollary}

\newtheorem{predf}[lemma]{Definition} 
\newenvironment{df}{\begin{predf}\rm}{\end{predf}}
\newtheorem{preremark}[lemma]{Remark}  
\newenvironment{remark}{\begin{preremark}\rm}{\end{preremark}}
\newtheorem{preremark0}[lemma]{Remark}  

\newtheorem*{prenotation}{Notation}

\numberwithin{equation}{section}

\newcommand\mnote[1]{}

\newcommand\mmnote[1]{}

\newcommand{\bR}{\mathbb{R}}

\newcommand\lra{\longrightarrow}

\newcommand\Emb{\mathrm{Emb}}

\newcommand\Th{\mathrm{Th}}

\newcommand{\map}{\mathrm{map}}

\def\Id{\mathrm{Id}}

\renewcommand{\geq}{\geqslant}
\renewcommand{\leq}{\leqslant}

\usepackage{comment}
\usepackage{color}
\usepackage{marginnote}

\def\g{\mathrm{Gauss}} 
\def\gauss{\mathrm{Gauss}}
\def\CL{\mathrm{CL}}
\def\graph{{\rm graph}}
\def\vol{\mathrm{vol}}
\def\graphvol{\graph\circ\nu}
\def\dvol{d_{\nu}}
\def\dgrascomp{d}
\def\dpsi{d_\psi}
\def\ls{\mathrm{ls}}
\def\gs{\mathrm{gs}}
\def\Gr{\mathrm{Gr}}

\def\Pr{\widetilde{\Psi}}
\def\fib{\mathrm{fib}}
\def\GL{\mathrm{GL}}

\def\s{\mathfrak{s}}
\def\scan{\mathcal{S}}
\def\st{\scan}
\def\Man{\mathbf{Man}}
\def\Top{\mathbf{Top}}
\def\Set{\mathbf{Set}}
\def\OO{\mathbf{O}}
\def\ss{\mathrm{ss}}
\def\ms{\mathrm{ms}}

\title{A metric for the space of submanifolds of Galatius and Randal-Williams}
\author{Federico Cantero Mor\'an}
\thanks{The author was funded by Michael Weiss' Humboldt professor grant and the Belgian Interuniversity Attraction Pole P07/18. He was partially supported by the Spanish Ministry of Economy and Competitiveness under grant MTM2013-42178-P}
\email{\texttt{federico.cantero@uclouvain.be}}
\address{{\normalfont IRMP, Chemin du cyclotron 2, 1348 Louvain-la-neuve, Belgique}}

\begin{document}
\begin{abstract}
Galatius and Randal-Williams defined a topology on the set of closed submanifolds of $\bR^n$ in \cite{GR-W}. B\"okstedt and Madsen in \cite{Bokstedt-Madsen} proved that a $C^1$ version of this topology is metrizable by showing that it is regular and second countable. Using that the scanning map of a topological sheaf on manifolds is an embedding, we give an explicit metric to the space considered by B\"okstedt and Madsen. Then, we compare this topology with the Fell topology and we use the Hausdorff distance to give another metric to the space of Galatius and Randal-Williams. 

\end{abstract}
\maketitle

\section{Introduction}
Let $U\subset \bR^n$ be an open subset and define $\psi(U)$ to be the set of all (possibly empty, possiby non-compact) proper submanifolds of $U$ of dimension $d$ without boundary. Here, a \emph{proper subset of $\bR^n$} is a subset of $\bR^n$ whose intersection with any compact subset is compact. A subset of $\bR^n$ is proper if and only if it is closed.
There is a surjective map
\[\coprod_{[N]} \Emb(N,U)\lra \psi(U)\]
that sends each embedding to its image (the union being indexed over diffeomorphism classes of manifolds of dimension $d$). Each choice of topology on the left hand-side (for instance, the Whitney $C^\infty$ topology or the $C^1$ topology) determines a quotient topology on the right hand-side. 

A drawback of all these topologies in $\psi(U)$ is that, while the assigment $U\mapsto \psi(U)$
is a sheaf of sets in $\bR^n$ (because the property of being a submanifold can be checked locally), the restriction maps of this sheaf are not continuous with respect to these topologies.

In order to make $U\mapsto \psi(U)$ a sheaf of topological spaces, Galatius and Randal-Williams \cite{GR-W} introduced a new topology on $\psi(U)$ which is reasonably close to the one coming from the Whitney $C^\infty$ topology in $\Emb(M,U)$. For example, both topologies coincide on the subset of $\psi(U)$ that consists of compact submanifolds. B\"okstedt and Madsen \cite{Bokstedt-Madsen} later proved that the $C^1$-version of the Galatius--Randal-Williams topology is metrizable. We recall now the definition of this topology: If $NW$ denotes the normal bundle of a submanifold $W$ of $U$, the exponential map is a partially defined function $\exp_W\colon NW\dasharrow U$
defined in a neighbourhood of the zero section $z$. Recall also the bundle projection 
\[\xymatrix{
\tau\colon TNW\ar[r]\ar[d] & TW\oplus NW\ar[r]\ar[d] & NW\ar[d] \\
NW\ar[r] & W\ar[r] & W.
}\]

\begin{df}\label{df:1} The \emph{($C^1$)-Galatius--Randal-Williams topology} on $\psi(U)$ is given by the following neighbourhood basis of any proper submanifold $W$:
\begin{itemize}
\item 
Every compact subset $K\subset U$ and every $\epsilon>0$ define a basic neighbourhood $(K,\epsilon)^{\gs}$ of $W$; a submanifold $W'$ belongs to $(K,\epsilon)^{\gs}$ if there is a section $f$ of the normal bundle $NW\to W$ such that 
\begin{enumerate}
\item\label{it:str} $\exp_W(f(W))\cap K = W'\cap K$ and 
\item\label{it:der} $\|f(x)\| + \|\tau\circ (Df)(x)\| < \epsilon$ for all $x\in W$ such that $\exp_W\circ f(x)\in W'\cap K$.
\end{enumerate}
\end{itemize}
\end{df}

We write $\Psi(U)$ for this topological space. Let $\Th^\fib(\gamma_d^\perp(T\bR^n))$ be the fibrewise one-point compactification of the fibrewise Grassmannian of affine $d$-planes in the tangent bundle of $\bR^n$. This is a bundle over $\bR^n$, and we write $\Gamma(\Th^\fib(\gamma_d^\perp(T\bR^n)))$ for its space of sections with the compact-open topology. Using a construction in \cite{cantero:merging}, we have a map
\[\st\colon \Psi(U)\lra \Gamma(\Th^\fib(\gamma_d^\perp(T\bR^n)))\]
and our first result shows that it is an embedding. As the right hand-side is a metric space, we have the first contribution of this note (Theorem \ref{thm:scan}), which gives a more conceptual proof of the forementioned result of B\"okstedt and Madsen: $\Psi(U)$ is a metric space, and an explicit metric is given by the restriction of the metric on $\Gamma(\Th^\fib(\gamma_d^\perp(T\bR^n)))$ along $\st$. We also generalise this result to spaces of submanifolds with labels in a topological space or in an abelian topological monoid (Theorem \ref{thm:scan2}).

There is another natural metric on $\psi(U)$: Consider the inclusion
\[\gauss\colon \psi(U)\lra \CL(U\times \Gr_d(\bR^n))\]
into the set of closed subsets of $U\times \Gr_d(\bR^n)$ that takes a submanifold to its tangent bundle. Endow the right hand-side with the Fell topology, which is induced by the Hausdorff metric on the set of closed subsets of the one-point compactification of $\bR^n\times \Gr_d(\bR^n)$. Finally, write $\widetilde{\Psi}(U)$ for $\psi(U)$ endowed with the subspace metric. In Theorem \ref{thm:ls}, we characterise this space and we show that the Fell topology on $\psi(U)$ is strictly coarser than the Galatius--Randal-Williams topology.

The metric obtained in Theorem \ref{thm:scan} is not easy to work with, because the map $\st$ involves several choices. We end this note with Theorem \ref{thm:gs}, giving a variation of the Hausdorff metric that induces the Galatius--Randal-Williams topology on $\psi(U)$.

Theorem~3.22 in \cite{GR-W} shows that when $\psi(\bR^n)$ is endowed with the Galatius--Randal-Williams topology, it has the weak homotopy type of $\Th(\gamma_d^\perp(\bR^n))$, whereas in \cite{cantero:merging} the author has proven that $\psi(\bR^n)$ with the induced Fell topology is weakly contractible. 

In Section \ref{s:2} we prove Theorems \ref{thm:scan} and \ref{thm:scan2}, in Section \ref{s:3} we give a discussion on topologies on spaces of closed subsets and state Theorems \ref{thm:ls} and \ref{thm:gs}, which are proven in Sections \ref{s:4} and \ref{s:5}. 

\subsection*{Acknowledgements} The author thanks Carles Casacuberta and Oscar Randal-Williams for comments on earlier versions of the paper, and he is specially grateful to Abd\'o Roig and his remarks.

\section{Metrics induced by the scanning map}\label{s:2}

Recall the site $\Man$ of $n$-dimensional smooth manifolds and open embeddings between them, and write $\Top$ for the category of topological spaces. Let $M$ be one such manifold, and let $g\colon M\to (0,\infty)$ be an injectivity radius for $M$ (i.e., the exponential map $\exp\colon T_pM\to M$ is injective on vectors of length at most $g(p)$). If $\Phi\colon \Man\lra \Top$ is a continuous functor, then the scanning map
\[\s\colon \Phi(M)\lra \Gamma(\Phi^\fib(TM))\]
is the adjoint of the map
\[\Phi(M)\times M\lra \Phi^\fib(TM)\]
that sends a pair $(x,p)$ to $\theta_{g(p)}^{-1}(\Phi(\exp_p)(x))$, where $\theta_t$ is a $\GL_n$-equivariant diffeomorphism from $\bR^n$ to the ball of radius $t$ in $\bR^n$. 

The assigment $U\mapsto \psi(U)$ defines a sheaf of sets on the site $\OO(\bR^n)$ of open subsets of $\bR^n$: if $U\subset V$, there is a function $\psi(V)\to \psi(U)$ given by sending each submanifold $W$ of $V$ to the submanifold $W\cap U$. In Section 2.2 of \cite{GR-W}, it is proven that the assignment $U\mapsto \Psi(U)$ is a topological sheaf.

Define $\psi(M)$ to be the set of all (possibly empty) proper $d$-dimensional submanifolds of $M$. This extends the sheaf of sets $\psi(-)$ from the site $\OO(\bR^n)$ of open subsets of $\bR^n$ to the site $\Man$. Consider now the diagram:
\[\xymatrix{
\OO(\bR^n)\ar[d]\ar[r]^{\Psi(-)} & \Top \ar[d] \\
\Man\ar[r]^{\psi(-)} & \Set
}\]
As explained in \cite[Theorem~3.3]{R-WEmbedded}, there is a unique lift
\[\Psi\colon \Man\lra \Top\]
 of $\psi(-)$, i.e., a topological sheaf on the site $\Man$ with underlying set-valued sheaf $\psi(-)$, whose restriction to $\OO(\bR^n)$ is $\Psi(-)$.

%
%
%
%

\begin{proposition} If $\Phi\colon \Man\to \Top$ is a topological sheaf, then the scanning map is an embedding. 
\end{proposition}
\begin{proof}
As the scanning map is a natural transformation between topological sheaves, it is an embedding if and only if its value on $\bR^n$ is an embedding. In this case, the scanning map is the adjoint of
\[\lambda\colon \Phi(\bR^n)\times \bR^n\lra \Psi(\bR^n),\]
where $\lambda_p$ is the translation that sends the origin to $p$ and $\lambda(x,p) = \lambda_p^{-1}(x)$. Now observe the following:

If $X\to \map(Y,Z)$ is injective, but not an embedding, then there is a coarser topology $X'$ on $X$ and a factorisation $X\to X'\to \map(Y,Z)$, and so a factorisation of its adjoint
\[X\times Y\lra X'\times Y\lra Z\]
where the first map is the identity on points.

Hence, if the scanning map were not an embedding, because $\bR^n$ is Hausdorff locally compact \cite[Theorem~46.11]{Munkres}, there would be a factorisation
\[\Phi(\bR^n)\times \bR^n\lra \Phi'(\bR^n)\times \bR^n\lra \Phi(\bR^n).\]
Its restriction is the identity
\[\Phi(\bR^n)\times \{0\}\lra \Phi'(\bR^n)\times \{0\}\lra \Phi(\bR^n),\]
and therefore $\Phi(\bR^n)\cong \Phi'(\bR^n)$, contradicting the hypothesis of the scanning map not being an embedding.
\end{proof}
The following is straightforward:
\begin{lemma} If $\Phi\colon \Man\to \Top$ is a topological sheaf and $\Lambda$ is a topological functor from the groupoid of vector spaces and isomorphisms, and $\sigma\colon \Phi(\bR^n)\lra \Lambda(\bR^n)$ is a $\GL_n$-equivariant quotient map such that the composition
\[\scan = \sigma\circ \s\colon \Phi(\bR^n)\lra \Gamma(\Phi^\fib(T\bR^n))\lra \Gamma(\Lambda^\fib(T\bR^n))\]
is injective, then the composition is an embedding.
\end{lemma}
Let us particularise taking $\Phi(V) = \Psi(V)$ and $\Lambda(V) = \Th(\gamma_d^\perp(V))$, the one-point compactification of the affine Grassmannian of $d$-planes in $V$. The space $\Lambda(V)$ includes $\GL_n$-equivariantly into the space $\Phi(V)$ as the subspace of possibly empty affine submanifolds. In \cite{cantero:merging} the author constructed a map
\[\Phi(V)\lra \Lambda(V)\]
which makes $\Lambda(V)$ a strong deformation retraction of $\Phi(V)$, and so this map is a right inverse of the inclusion, and therefore a quotient map. This map is easily seen to be $\GL_n$-equivariant. Finally, the composition
\[\st\colon\Phi(\bR^n)\lra \Gamma(\Phi^\fib(T\bR^n))\lra \Gamma(\Lambda^\fib(T\bR^n))\]
is injective, because if $W\in \Phi(\bR^n)$, then $\st(W)^{-1}(\Gr_d^\fib(T\bR^n)) = W$, and so $\st(W)\neq \st(W')$ if $W\neq W'$. Therefore, we have that:
\begin{theorem}\label{thm:scan} The map $\st$ is an embedding, and therefore $\Psi(M)$ is metrizable, and an explicit metric is the pullback of any metric on $\Gamma(\Lambda^\fib(T\bR^n))$.
\end{theorem}

\subsection*{Other metrizable spaces of submanifolds} Let $\widetilde{\psi}(U;X)$ be the set of pairs $(W,\alpha)$ where $W\in \psi(U)$ and $\alpha$ is a continuous function from $W$ to a metric abelian topological monoid $X$.

\begin{df}\label{df:2} The topology on $\widetilde{\Psi}(U;X)$ is given by the following neighbourhood basis of any pair $(W,\alpha)$:
Every compact subset $K\subset U$ and every $\epsilon>0$ define a basic neighbourhood $(K,\epsilon)^{\ms}$ of $(W,\alpha)$; a pair $(W',\alpha')$ belongs to $(K,\epsilon)^{\ms}$ if there is a subset $Q\subset NW$ such that the composite $q\colon Q\subset NW\to W$
\begin{enumerate}
\item $\exp_W(Q)\cap K = W'\cap K$,
\item $\|f(x)\| + \|\tau\circ (Df)(x)\| < \epsilon$ for each local section $f$ of $q$.
\item 
\[d\left(\alpha(x),\sum_{y\in q^{-1}(x)}\alpha'(y)\right)<\epsilon\text{ for all }x\in W\cap K,\]
\end{enumerate}
\end{df}
Let $\psi(U;X)$ be the set of pairs $(W,\alpha)$ where $W\in \psi(U)$ and $\alpha$ is a continuous function from $W$ to a topological space $X$.

\begin{df}\label{df:3} The topology on $\Psi(U;X)$ is given by the following neighbourhood basis of any pair $(W,\alpha)$: 
Every compact subset $K\subset U$, every $\epsilon>0$ and every neighbourhood $A$ of $\alpha\in \map(W,X)$ define a basic neighbourhood $(K,\epsilon,A)^{\ss}$ of $(W,\alpha)$; a pair $(W',\alpha')$ belongs to $(K,\epsilon)^{\ss}$ if there is a section $f$ of the normal bundle $NW\to W$ such that:
\begin{enumerate}
\item $\exp_W(Q)\cap K = W'\cap K$,
\item $\|f(x)\| + \|\tau\circ (Df)(x)\| < \epsilon$ for all $x\in W$ such that $\exp_W\circ f\in W'\cap K$.
\item $\alpha'\circ \exp_W\circ f\in A$,
\end{enumerate}
\end{df}

%
These assigments define topological sheaves on the site of open subsets of $\bR^n$ (see \cite[Lemma~3.2]{cantero:merging} for $\widetilde{\Psi}(\bR^n,X)$, whose proof also works for $\Psi(\bR^n,X)$), which, as in the beginning of this section, extend to the site of $n$-dimensional manifolds and open embeddings. The main result of \cite{cantero:merging} generalises to give $\GL_n$-equivariant right inverses of the inclusions
\begin{align*}
\widetilde{\Lambda}(\bR^n,X)&\hookrightarrow \widetilde{\Psi}(\bR^n,X)\\
\Lambda(\bR^n,X)&\hookrightarrow \Psi(\bR^n,X)
\end{align*}
of the subspace of those $(W,\alpha)$ where $W$ is a (possibly empty) union of parallel planes in $\bR^n$ and $\alpha$ is locally constant. These subspaces are easily seen to be metrizable if $X$ is metrizable, therefore
\begin{theorem}\label{thm:scan2} If $X$ is metrizable, then the spaces $\widetilde{\Psi}(M,X)$ and $\Psi(M,X)$ are metrizable, and an explicit metric is obtained pulling back the metric of $\Gamma(\widetilde{\Lambda}(-,X)^\fib(TM))$ and $\Gamma(\Lambda(-,X)^\fib(TM))$ along $\scan$.
\end{theorem}

\section{Topologies and distances on spaces of closed subsets}\label{s:3}
We start introducing two topologies in the set $\CL(X)$ of closed subsets of a space $X$. If $U\subset X$, define
\[U^-= \{A\in \CL(X)\mid A\cap U\neq \emptyset\}, \quad U^+=\{A\in \CL(X)\mid A\subset U\}.\]
Then the \emph{Fell topology} has as subbasis the collection of all subsets $U^-$ with $U$ an open subset of $X$ and $U^+$ with $U$ the complement of a compact subset of $X$ \cite{Fell}. The \emph{finite topology} has as subbasis the collection of all subsets $U^-$ and $U^+$ with $U$ an open subset \cite{Michael}. 

If $X$ is a uniformly hemicompact\footnote{It is a countable union $\bigcup_{k=1}^\infty H_k$ of compact subsets, where each subset $H_{k+1}$ contains a neighbourhood of $H_{k}$.} metric space, then the Fell topology on $\CL(X)$ is metrizable, and an explicit metric is the following: If $\overline{X}$ denotes the one-point compactification of $X$, then there is a map
\[\CL(X)\lra \CL(\overline{X})\]
that sends a closed subset $A$ to the closed subset $A\cup\{\infty\}$. This map is an embedding (it even admits a retraction). Since $X$ is metrizable and uniformly hemicompact, $\overline{X}$ is metrizable \cite{Mandelkern, metric-mathoverflow}. It is immediate to see that the Fell topology in a compact space agrees with the finite topology, which for compact metric spaces is induced by the Hausdorff distance \cite{Michael}.

\begin{df} The \emph{differential Fell topology} on $\psi(U)$ is the restriction of the Fell topology along the inclusion 
\[\gauss\colon \psi(U)\lra \CL(U\times \Gr_d(\bR^n)).\]
We let $\dgrascomp_H$ be the metric in $\psi(U)$ induced by the Hausdorff metric along the inclusion
\[\psi(U)\lra \CL(U\times \Gr_d(\bR^n))\lra \CL(\overline{U\times\Gr_d(\bR^n)}).\]
\end{df}

One might expect the distance $\dgrascomp_H$ to be a metric for $\Psi(U)$, but it is not. To show why (see also Figure \ref{figure:psi}), we take a compact submanifold $W$. A submanifold $W'$ is close to $W$ in $\Psi(U)$ if there is a \emph{global} section of the normal bundle of $W$ that is close to the zero section and $W' = f(W)$. On the other hand, $W'$ is close $W$ in $(\psi(U),\dgrascomp_H)$ if each point $(x,T_xW)$ is close to $\gauss(W')$ and each point $(y,T_yW')$ is close to $\gauss(W)$. Therefore, if we let $\delta>0$, we let $s$ be a section of $NW$, and we identify $NW$ with a tubular neighbourhood of $W$ in $U$, and let $W'_\delta = \delta\cdot s(W)\cup -\delta\cdot s(W)$, then $\dgrascomp_H(W,W')<\delta$, hence as $\delta\to 0$, the sequence of submanifolds $W'_\delta$ (which are diffeomorphic to two disjoint copies of $W$), converges to $W$. But this is not allowed in Definition \ref{df:1} ($W'_\delta$ is never the image of a global section). Nevertheless, we will prove that any $W'$ in a small neighbourhood of $W$ (with the differential Fell topology) is always \emph{locally} the image of some local sections.

\begin{figure}[b]
\centering
\subfloat[]{\label{gs}\includegraphics[width = 4cm]{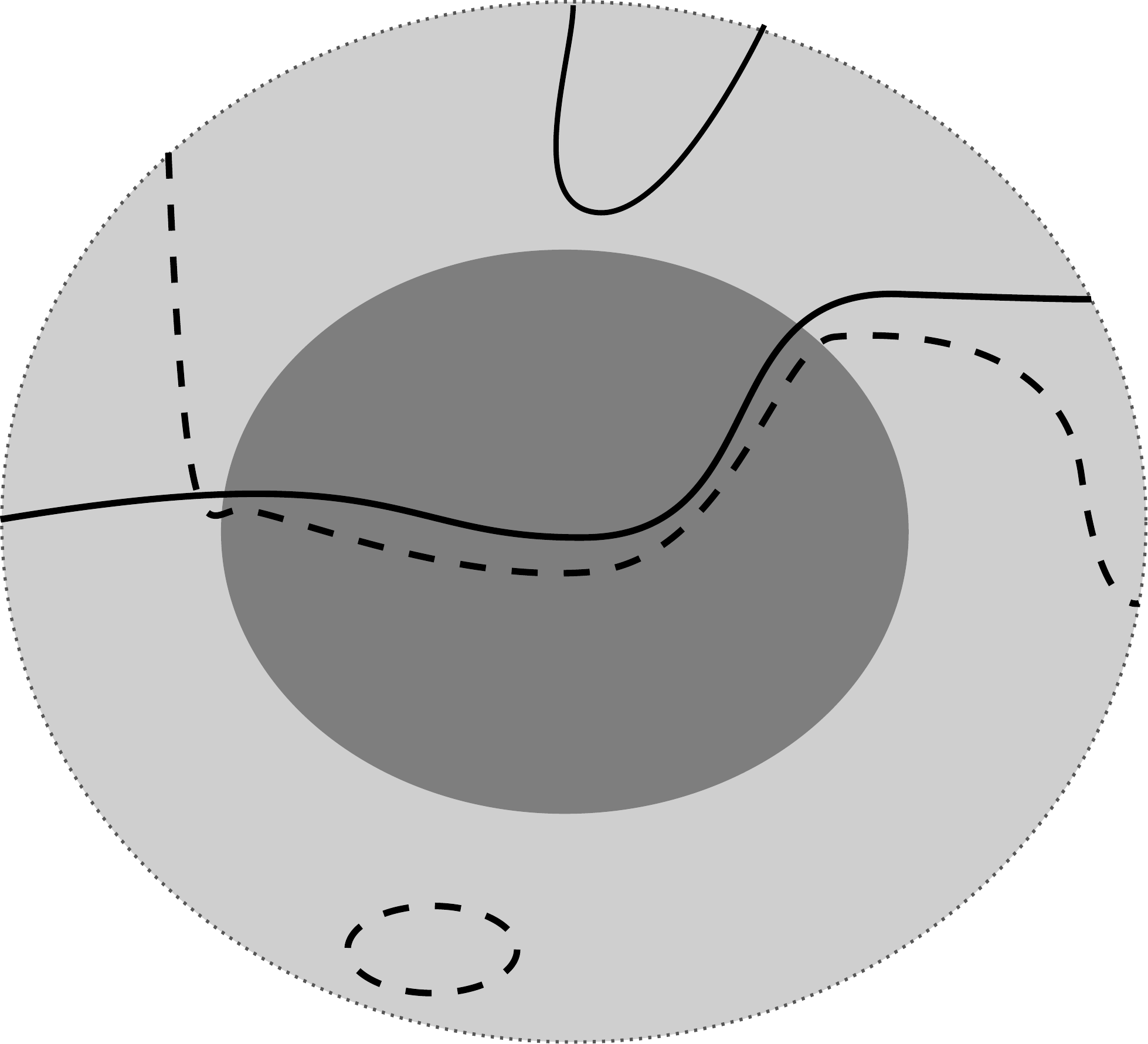}} \hspace{1cm}
\subfloat[]{\label{ls}\includegraphics[width = 4cm]{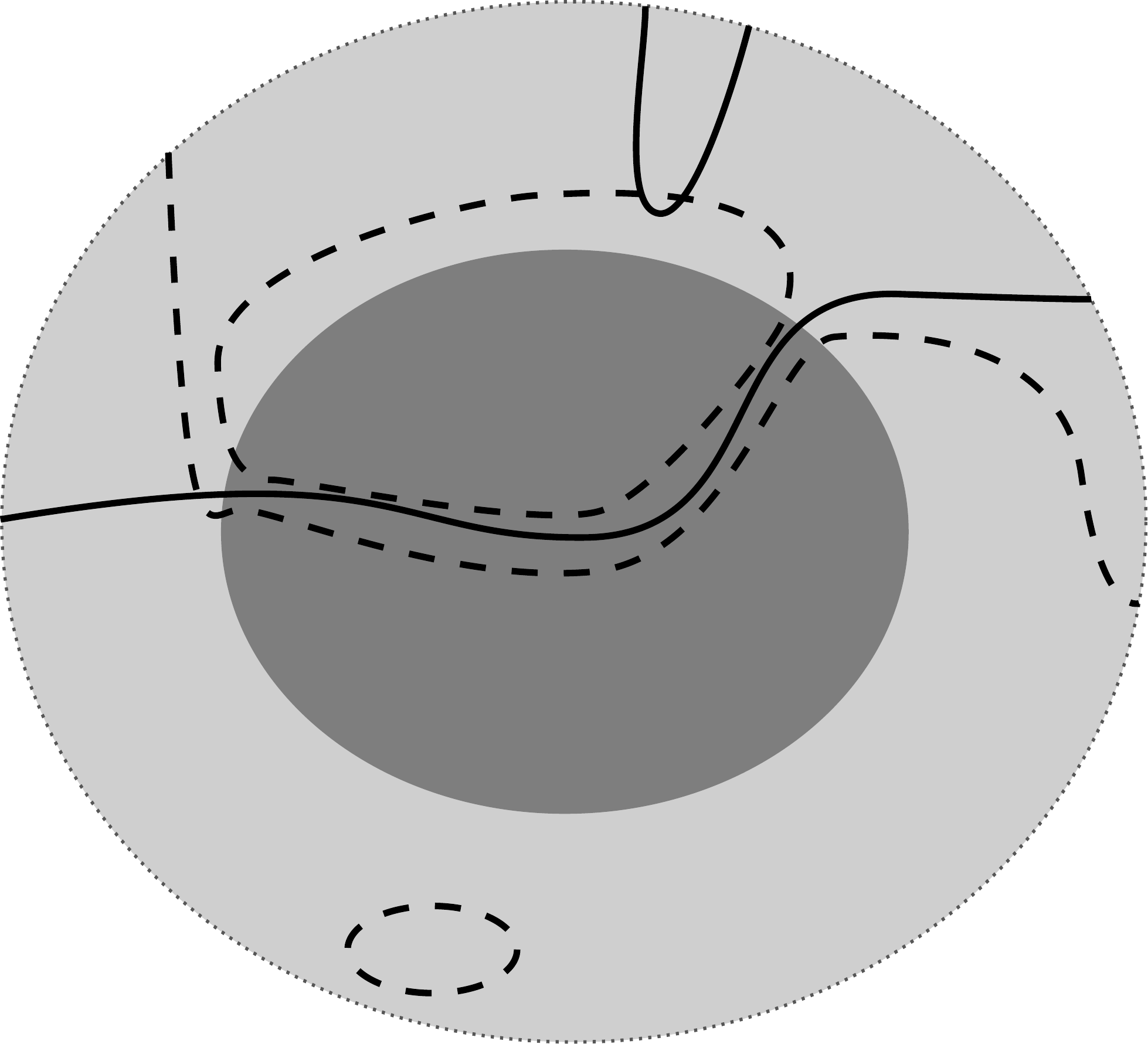}} \\
\subfloat[]{\label{lejos}\includegraphics[width = 4cm]{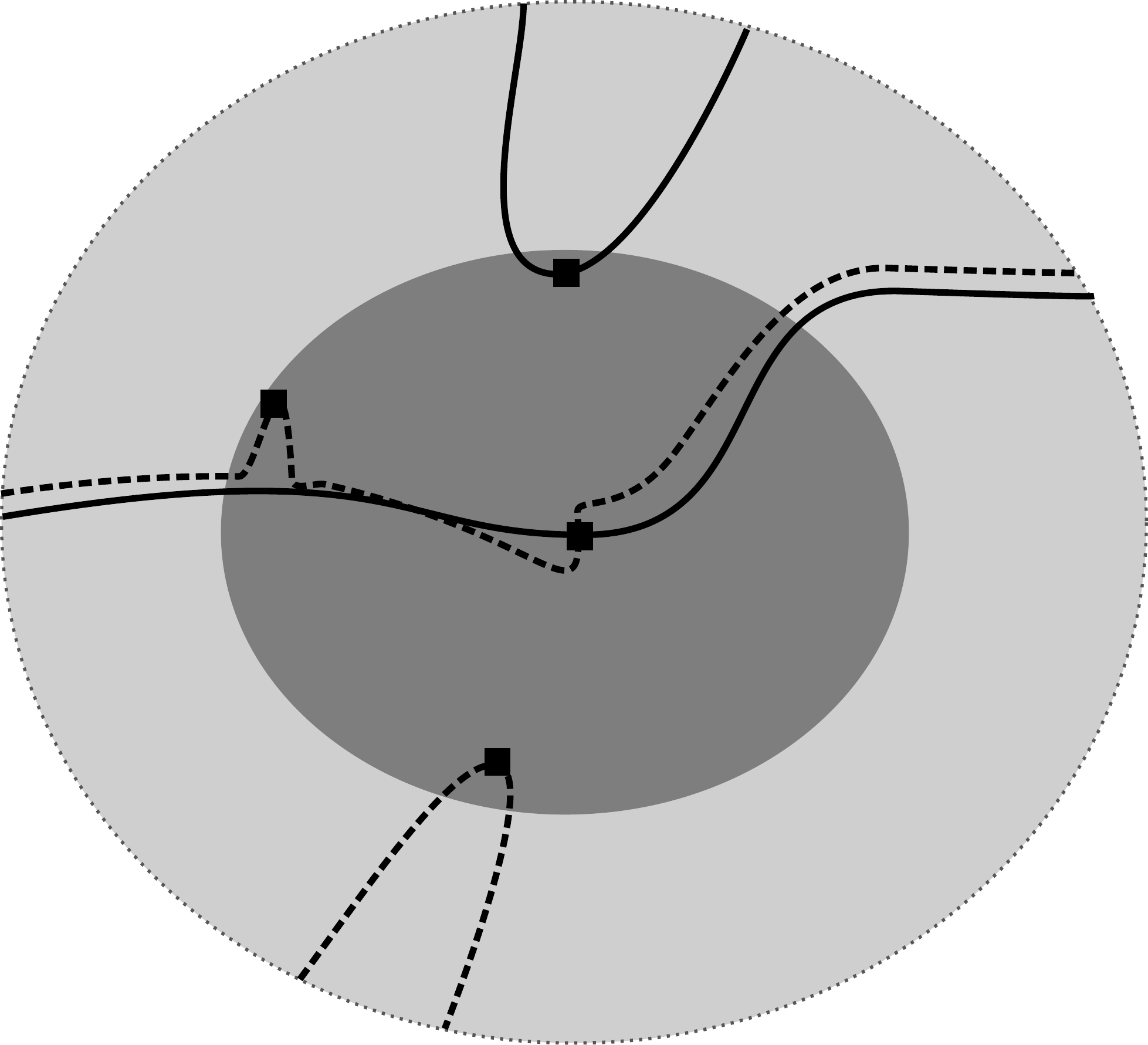}}\hspace{1cm}
\subfloat[]{\label{empty}\includegraphics[width = 4cm]{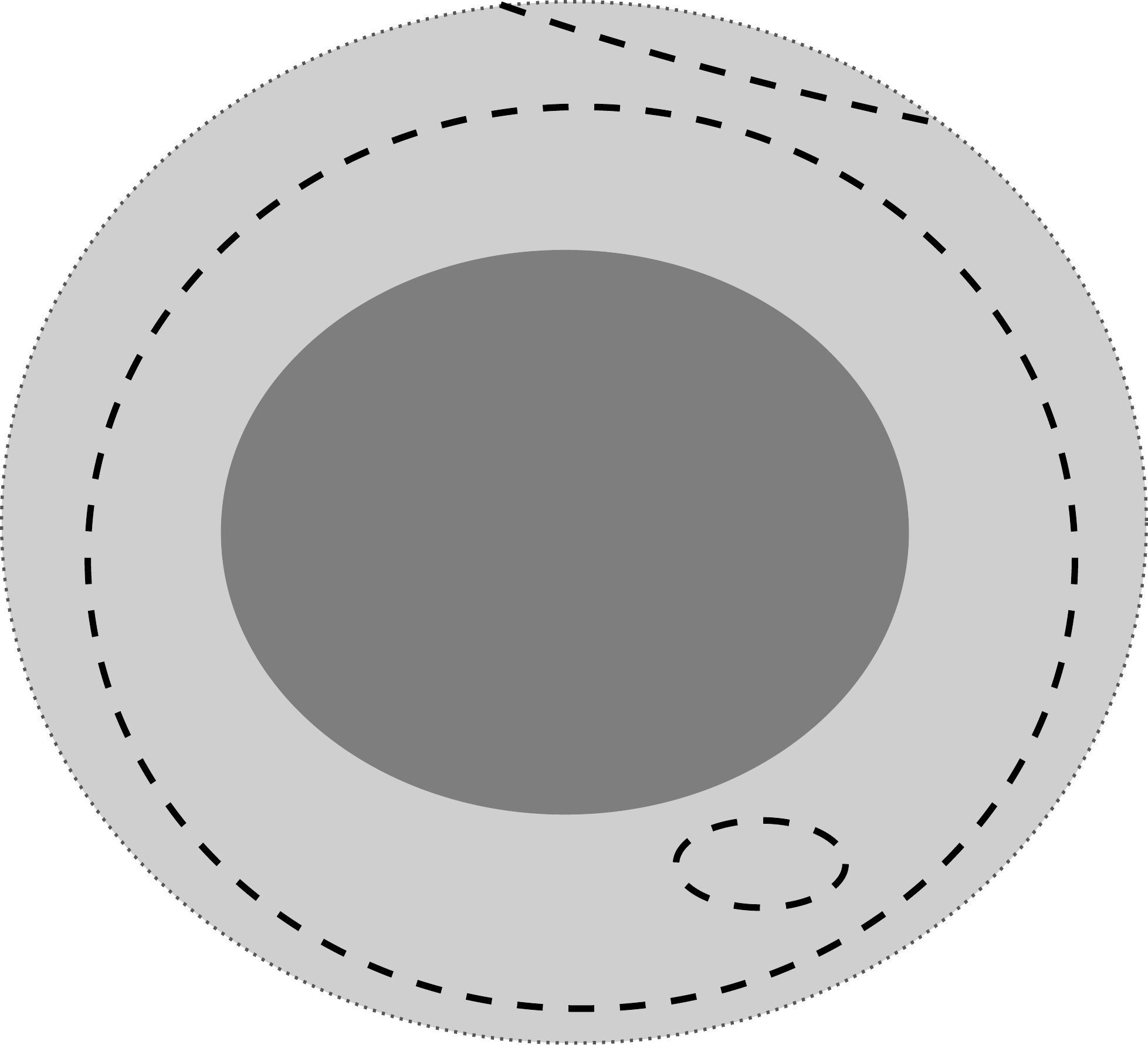}}
\caption{The light grey area is an open subset $U$ while the dark area is a compact subset $K\subset U$. We denote the dashed submanifold by $W'$ and the non-dashed submanifold by $W$. In Figure \ref{gs}, $W'$ is close to $W$ both in $\Psi(U)$ and $\Pr(U)$. In Figure \ref{ls}, $W'$ is close to $W$ in $\Pr$ but not in $\Psi(U)$. In Figure \ref{lejos}, $W'$ is far from $W$ in $\Pr(U)$, hence in $\Psi(U)$ too. The small squares indicate points that are far from the other submanifold. In Figure \ref{empty}, $W'$ is close to $\emptyset$.}
\label{figure:psi}
\end{figure}

In order to make this explicit, we introduce now a different topology in $\psi(U)$, the only difference being that instead of requiring $W'\cap K$ to be the image of a global section of $NW$, we only ask it to be the union of images of local sections of $NW$ whose domains cover $W$.

\begin{df}\label{def:ls} The space $\Pr(U)$ has the same underlying set as $\Psi(U)$, with neighbourhood basis of a proper submanifold $W$:
\begin{itemize}
\item 
Every compact subset $K\subset U$ and every $\epsilon>0$ define a basic neighbourhood $(K,\epsilon)^{\ls}$ of $W$; a submanifold $W'$ belongs to $(K,\epsilon)^{\ls}$ if there is a subset $Q\subset NW$ such that the composite $q\colon Q\subset NW\to W$
\begin{enumerate}
\item hits every point of $W\cap K$,
\item is a local diffeomorphism, i.e., a covering map,
\item\label{qui2} $\exp_W(Q)\cap K = W'\cap K$,
\item\label{bla2} $\|f(x)\| + \|\tau\circ (Df)(x)\| < \epsilon$ for each local section $f$ of the local diffeomorphism $q$.
\end{enumerate}
\end{itemize}
\end{df}

\begin{theorem}\label{thm:ls} The space $\Pr_d(U)$ has the differential Fell topology (which therefore is strictly coarser than the Galatius--Randal-Williams topology).
\end{theorem}

Recall that a \emph{pseudo-metric} on a set $X$ is a symmetric function $d\colon X\times X\to [0,\infty)$ satisfying the triangle inequality and such that $d(x,x)=0$ for all $x\in X$. The balls of a pseudo-metric on $X$ define a basis of a topology, and that topology is Hausdorff if and only if $d$ is a metric (that is, $d(x,y)\neq 0$ whenever $x\neq y$). If $(X,d)$ is a metric space and $f\colon Y\rightarrow X$ is a function from a set $Y$, then $d\circ(f\times f)$ is a pseudo-metric on $Y$, which is a metric if $f$ is injective. 

Let $F_0(\bR_+)$ be the set of non-decreasing functions from $\bR_+=[0,\infty)$ to itself that preserve $0$. Define
\begin{align*}\label{eq:q}
\graph\colon F_0(\bR_+) &\lra \CL(\bR_+\times\bR_+) \\
f &\longmapsto \{(x_1,x_2)\in \bR_+\times\bR_+\mid f(x_1^-)\leq x_2\leq f(x_1^+)\}
\end{align*}
Now let $W_r$ be the intersection of $W$ with the closed ball of radius $r$. Define
\begin{align*}
\nu\colon \psi(\bR^n)&\lra F_0(\bR_+)\\
W&\longmapsto (r\mapsto \vol(W_r))
\end{align*}
\begin{df} We write $\dvol$ for the restriction of the Hausdorff distance associated to the Fell topology in $\CL(\bR_+\times \bR_+)$ along $\graph\circ \nu$. Because $\nu$ is not injective, $\dvol$ is only a pseudo-distance.
\end{df}
\begin{df} Define the following metric in $\psi(\bR^n)$
\[\dpsi(W,W') = \dgrascomp_H(W,W') + \dvol(W,W').\]
\end{df}

\begin{theorem}\label{thm:gs} The distance $\dpsi$ is a metric for the space $\Psi(\bR^n)$.
\end{theorem}
\begin{remark}The space $\psi(U)$ can be naturally endowed with a metric too, following the same method as above with
\[W_r = \left\{x\in W\left| \max\left\{\|x\|,\frac{1}{d(x,\bR^n\setminus U)}\right\}\leq r\right.\right\}.\]
\end{remark}

\section{ The space $\Pr(U)$ and the $\dgrascomp_H$-topology}\label{s:4}
In Lemmas \ref{lemma:ls1} and \ref{lemma:ls2} of this section we prove Theorem \ref{thm:ls}. We start setting up some conventions.

We denote by $d_0$ the Euclidean distance in $\bR^n$, by $d_1$ the distance on $\Gr_d(\bR^n)$ given by
\[d_1(L,L') = \max_{v\in S(L)}\min_{w\in S(L')}\{\mathrm{angle}(v,w)\} = \max_{w\in S(L')}\min_{v\in S(L)}\{\mathrm{angle}(v,w)\},\]
where $S(L)$ is the unit sphere in $L$, and by $d = d_0+d_1$ the distance in $\bR^n\times \Gr_d(\bR^n)$. If $f\colon L\to L'$ is a linear operator, then 
\[\|f\| = \max_{v\in S(L)}\{\|f(v)\|\}.\]

If $W\in \psi(U)$, we denote by $p\colon NW\to W$ the projection, which is covered by two bundle maps: its differential $Dp\colon T(NW)\to TW$ and the canonical bundle isomorphism
\[\xymatrix{
TNW \ar[r]^-\alpha\ar[d] & TW\oplus NW\ar[d] \\
NW\ar[r]^p & W.
}\]
The differential $Dp$ is the composition of $\alpha$ with the projection onto $TW$, and we write $\tau$ for the composition of $\alpha$ with the projection onto $NW$.

Recall that there is a function $\epsilon\colon W\to (0,\infty)$ such that the restriction
\[\exp^\epsilon_W\colon N^\epsilon W\longrightarrow U\]
 of the exponential map $\exp_W\colon NW\rightarrow U$ to the subspace $N^\epsilon W$ of vectors $v$ of length at most $\epsilon(p(v))$ is an embedding. In addition, the restriction of this map to each fibre 
\[\exp_{W,x}^\epsilon\colon N_x^\epsilon W\to U\]
is a radial isometry. We denote by $z\colon W\to NW$ the zero section, by $T^\epsilon$ the image of $N^\epsilon W$, and by $\pi$ the composition of $(\exp^{\epsilon}_W)^{-1}$ and $p$. Summarizing:
\[\xymatrix{
NW \ar @{}[r]|{\supset}\ar[d]_p & N^\epsilon W\ar[d]_{\exp_W^\epsilon}^\cong & \\
W\ar @{}[r]|{\subset} \ar @/^/[ur]^{z} &T^\epsilon\ar @{}[r]|{\subset} \ar @/^/[l]^{\pi} & U
}\]

%
%

\begin{lemma}\label{lema2} Let $A$ be an open subset of $W$, let $x\in A$ and let $f\colon A\to NW$ be a local section of $N W$. Let $A'=(\exp_W\circ f)(A)$ and let $y = (\exp_W\circ f)(x)$. Then 
\begin{align*}
\|f(x)\| &= d_0(x,y) \\
\|\tau\circ (Df)(x)\| &= \tan(d_1(T_xA,T_yA')).
\end{align*}
\end{lemma}
\begin{proof} The first equality follows because $\exp_W$ is a radial isometry. For the second, observe that if $\rho$ is a linear projection of Euclidean spaces and $\|v\|=1$, then
\[\|\rho(v)\| = \tan(\mathrm{angle}(v,\rho^\perp(v))).\]
Now,
\begin{align*}
&\|\tau\circ (Df)(x)\|= \max_{v\in S(T_xA)}{\|\tau\circ (Df)(x)(v)\|}= &&\\
&= \max_{v\in S(T_xA)}{\tan\left(\mathrm{angle}(Df(x)(v),\tau^\perp Df(x)(v))\right)} &\text{by the observation before}\\
&= \max_{v\in S(T_xA)}{\tan\left(\mathrm{angle}(Df(x)(v),v)\right)} & \tau^\perp (Df) = (Dp)(Df)=D\Id \\
&= \max_{w\in S(T_{f(x)}f(A))}\min_{v\in S(T_xA)}{\tan\left(\mathrm{angle}(w,v)\right)} &\text{because $\tau$ is a projection} \\
&= \tan\left(\max_{w\in S(T_{f(x)}f(A))}\min_{v\in S(T_xA)}{\mathrm{angle}(v,w)}\right) &\\
&= \tan\left(d_1(T_xA,T_{f(x)}f(A))\right) &\\
&= \tan\left(d_1(D\exp_W(T_xA),D\exp_W(T_{f(x)}f(A)))\right) &\text{as $\exp_W$ is a radial isometry}\\
&= \tan\left(d_1(T_xA,T_{y}A')\right).&\qedhere
\end{align*}
\end{proof}

\begin{lemma}\label{lemma:ls1} Let $W\in \psi(U)$ and let $(K,\epsilon)^{\ls}$ be a neighbourhood of $W$ in $\Pr_d(U)$. Then, there is a neighbourhood $V$ of $W$ in the differential Fell topology with $V\subset (K,\epsilon)^{\ls}$. 
\end{lemma}
\begin{proof}
Let $T$ be a tubular neighbourhood contained in $\exp_W(N^{\epsilon/2}W)$. We start defining the following subsets:
\begin{itemize}
\item $T'$ is the closure of the relatively compact subset $\pi^{-1}\pi(T\cap K)$
\item $K'$ is the closure of the complement of $T'$ inside $K\cup T'$, which is compact.
\item $K_1 = K'\times \Gr_d(\bR^n)$, which is compact.
\item $K_2=\{(x,L)\in \bR^n\times \Gr_d(\bR^n)\mid x\in T', L\perp T_{\pi(x)}W\}$, which is compact.
\item $K_{3,\epsilon}=\{(x,L)\in \bR^n\times \Gr_d(\bR^n)\mid x\in T', d(L,T_{\pi(x)}W)\geq \arctan(\epsilon/2)\}$, which is compact. 
\item $A_i$ is the complement of $K_i$, for $i=1,2,3$. 
\item $\{B_j\}$ is a collection of connected open subsets of $T'$ such that each $B_j$ is connected and the inclusion $\bigcup_j {B_j} \to T'$ is surjective on components.
\end{itemize}
Define $V = A_1^+\cap A_2^+\cap A_{3,\epsilon}^+\cap B^-$ (observe that $K_2\subset K_{3,\epsilon}$, hence $A_{3,\epsilon}^+\subset A_2^+$, so we may remove $A_2^+$, but we keep it there to help the next explanation), and let $W'\in V$. We claim that $W'\in (K,\epsilon)^\ls$. To prove this, define $Q = \exp_W^{-1}(W'\cap T')$. Then,
\begin{itemize}
\item Because $W'\in A_1^+$, it follows that $W'\cap K = W'\cap T'$, hence  
\[\exp_W(Q)\cap K := W'\cap T'\cap K = W'\cap K.\]
\item Since $W'\in A_2^+$, it follows that the restriction of $p\colon NW\to W$ to $Q$ is a submersion onto its image, which is proper because $Q$ is compact.
\item Since $W'\in B^-$, it follows that this proper submersion has non-empty fibres over each component and therefore it is a surjective submersion.
\item $\|f(x)\| + \|\tau\circ (Df)(x)\|<\epsilon$ because by Lemma \ref{lema2}:
\begin{itemize}
\item since $Q\subset T'\subset T \subset \exp_W(N_{\epsilon/2}W)$, it follows that $\|f(x)\|<\epsilon/2$ for every local section $f$ of $q$,
\item since $W'\subset A_{3,\epsilon}^+$, 
\[\|\tau\circ(Df)(x)\| = \tan(d_1(T_xW,T_{(\exp_W\circ f)(x)}W'))<\tan(\arctan(\epsilon/2))<\epsilon/2.\qedhere\]
\end{itemize} 
\end{itemize}
\end{proof}


\begin{lemma}\label{lemma:ls2} Let $W\in \psi(U)$ and let $V$ be a subbasic neighbourhood of $W$ in the differential Fell topology. Then, there is a neighbourhood $(K,\epsilon)^{\ls}$ of $W$ in $\Pr_d(U)$ with $(K,\epsilon)^{\ls}\subset V$.
\end{lemma}
\begin{proof}
Suppose first that $N$ is a subbasic neighbourhood of $W$ of the form $A^+$, with $A\subset \bR^n\times \Gr_d(\bR^n)$ and the complement of $A$ a compact subset $J$. This means that $\g(W)\cap J = \emptyset$. Let $K$ be the image of the projection of $J\subset \bR^n\times \Gr_d(\bR^n)$ onto $\bR^n$. For each $x\in W\cap K$, define 
\[\epsilon_x = \frac{1}{2}d((x,T_xW),J),\]
and let $\epsilon$ be the minimum of all $\epsilon_x$. 

Now let $W'\in (K,\epsilon)^{\ls}$ and let $Q\subset NW$ be as in Definition \ref{def:ls}. If $y\in W'\cap K$ there is a section $f$ of $q\colon Q\to W$ whose target contains a neighbourhood of $y$. We have, using Lemma \ref{lema2}: 
\begin{align}\label{eq:110}
\begin{split}
d((x,T_xW),(y,T_yW')) &= d_0(x,y) + d_1(T_xW,T_yW') \\
&= \|f(x)\| + \arctan(\|\tau\circ (Df)(x)\|) \\
&\leq \|f(x)\| + \|\tau\circ (Df)(x)\| < \epsilon.
\end{split}
\end{align}
Now we can compute
\begin{align*}d((y,T_yW'),J)&\geq d((x,T_xW), J) - d((x,T_xW),(y,T_yW')) \\
&> d((x,T_xW), J) - \epsilon  \geq \frac{1}{2}d((x,T_xW), J)>0,\end{align*}
therefore $\g(W')\cap J=\emptyset$. 


Now, let $B^-$ be a subbasic neighbourhood of $W$ with $B\subset \bR^n\times \Gr_d(\bR^n)$ an open subset. Let $x\in W\cap \pi(B)$, let $\epsilon$ be such that the ball with centre $(x,T_xW)$ and radius $\epsilon$ is contained in $B$, and let $K$ be a compact neighbourhood of $x$ in $\bR^n$.

Let $W'\in (K,\epsilon)^\ls$, so that there exists a $Q\subset NW$ satisfying the conditions of Definition \ref{def:ls}. In particular, $q\colon Q\to W\cap K$ is surjective, so we may pick a $y\in q^{-1}(x)$. But as in \eqref{eq:110}, we have $d((y,T_yW'),(x,T_xW))<\epsilon$, so $(y,T_yW')\in B$ too, hence $W'\in B^-$.
\end{proof}


\section{ The space $\Psi(\bR^n)$ and the $\dpsi$-topology}\label{s:5}
In Lemmas \ref{lemma:gs1} and \ref{lemma:gs2} of this section we prove Theorem \ref{thm:gs}.
\begin{lemma}\label{lemma:subbasic} The subbasic elements $D^+$ (where $D$ is any open subset containing $(s,\infty)\times \bR_+$ for some $s\geq 0$) generate the topology in $F_0(\bR_+)$ given by the restriction of the Fell topology along the inclusion $\graph\colon F_0(\bR_+)\hookrightarrow \CL(\bR_+\times\bR_+)$. 
\end{lemma}
\begin{proof}
First we show that this new topology is finer than the one induced by the Fell topology. 

Let $D$ be an open subset of $\bR_+\times \bR_+$ that intersects non-trivially $\graph(f)$, i.e., $\graph(f)\in D^-$. Since $A\subset B$ implies that $A^-\subset B^-$, we may assume that $D$ is bounded. 

If $0\in D$, then $D^-$ is the whole space. 

If $0\notin D$, take $f$ with $\graph(f)\in D^-$. Let $x\in \graph(f)\cap D$, and let $\ell$ be the line with slope $-1$ that goes through $x$. Let $A,B$ be the two components of $\bR_+\times\bR_+ \setminus \ell$, with $A$ the component that contains $0$. Then $A\cup B\cup D$ is a neighbourhood of $\graph(f)$ (here it is important that $f$ is non-decreasing) and its complement is compact. We will prove that $(A\cup B\cup D)^+\subset D^-$. 

Suppose that $\graph(g)\in (A\cup B\cup D)^+$, i.e., $\graph(g)\subset A\cup B\cup D$. Then:
\begin{itemize}
\item $\graph(g)\cap A \neq \emptyset$ because both contain the origin.
\item $\graph(g)\cap B \neq \emptyset$ because $\graph(g)$ is unbounded and both $A$ and $D$ are bounded. 
\end{itemize}
On the other hand, $\graph(g)\not\subset A\cup B$, because being $A$ and $B$ disjoint and $\graph(g)$ connected, it would imply that either $\graph(g)\subset A$ or $\graph(g)\subset B$, which we have proven to be false. As a consequence, $\graph(g)\cap D\neq \emptyset$, so $\graph(g)\in D^-$.

Second, we show that this new topology is coarser than the one induced by the Fell topology.

Let $D$ be an open subset of $\bR_+\times \bR_+$ containing $(s,\infty)\times \bR_+$ for some $s\geq 0$. Take $f$ such that $\graph(f)\in D^+$. Let $n>\max_{r\in [0,s]}\{f(r)\}$. Let $$A=(D\cup (n+1,\infty))\setminus [0,s]\times [n,n+1].$$ 
Then $A$ is the complement of a compact subset, hence $A^+$ is a neighbourhood of $f$ in the Fell topology which is contained in $D^+$. 
\end{proof}

\begin{lemma}\label{lemma:ls} Let $W\in \psi(U)$, let $s>0$, 
and let $V\subset [0,\infty)\times \bR$ be a neighbourhood of the graph $\graphvol(W)$ that constains $(s,\infty)\times \bR_+$. Then there exists a neighbourhood $(K',\epsilon')^\ls$ of $W$
, such that if $W'\in (K',\epsilon')^\ls$, then:
\begin{align*}
\graphvol(W) := \graph(r\mapsto \vol(W_r))  \in (c\cdot V)^+.
\end{align*}
\end{lemma}
\begin{proof} As a first step in the proof, we take $V$ a neighbourhood of the graph $\graphvol(W)$ that contains $(s+1,\infty)\times \bR_+$ and such that $V'$ is a neighbourhood of its closure. We will produce a neighbourhood $(K,\epsilon)^\ls$ of $W$ such that if $W'\in (K,\epsilon)^\ls$, then
\[\graph(r\mapsto \vol(\pi(W'_r)) \subset V.\]
Let $K$ be the disc of radius $s+2$, let $\delta$ be such that 
\[d(x,\graphvol(W))<\delta\Rightarrow x\in V.\] 
Let $\epsilon$ be so small that 
\[\vol(W_{r+\epsilon}\setminus W_r), \vol(W_r\setminus W_{r-\epsilon}) <\delta, \forall r<s.\]
If $W'\in (K,\epsilon)^\ls$, we have that for each $y\in W'$ the euclidean distance between $y$ and $\pi(y)$ is bounded by $\epsilon$. Therefore, $W_{r-\epsilon}\subset \pi(W_r)\subset W_{r+\epsilon}$ for all $r<s$, hence
\[\vol(W_r)-\delta < \vol(W_{r-\epsilon})<\vol(\pi(W'_r))<\vol(W_{r+\epsilon})<\vol(W_r)+\delta\]
and so $\graphvol(\pi(W'_r))\subset V$. 

For the second step, take $\mu$ to be the distance from the closure of $V$ to the complement of $V'$. Then we will provide a neighbourhood $(K'',\epsilon'')^\ls$ of $W$ such that if $W'\in (K'',\epsilon'')$, then
\begin{align*}
\left|\vol(W'_r) - c\cdot\vol(\pi(W'_r))\right|&<\mu, \forall r<s.
\end{align*}
Then, the proof will be finished taking $(K',\epsilon') = (K\cup K'',\min(\epsilon,\epsilon''))$: by the first step of this lemma, $\graphvol(\pi(W'_r))\subset V$ and by the second step, 
\[d(\graphvol(W'_r),V)<d(\graphvol(W'_r),\graphvol(\pi(W'_r)))<\mu,\]
 hence $\graphvol(W'_r)\subset V'$.

Let $\pi_r$ be the restriction of $\pi$ to $W'_r$ and let $W(r)\subset W$ be its image. Since $\pi_r\colon W_r\to W(r)$ is a local diffeomorphism for all $r<s$, we may choose, for each $x\in W'_r$, a local section $f_x$ of $\pi_r$. Writing $J_y$ for the Jacobian at point $y$,
\[\int_{x\in W'_r} 1dx = \int_{y\in W(r)} \sum_{\pi(x)=y}f_x^*(dy) = \int_{y\in W(r)}\sum_{\pi(x)=y}\det(J_y(f_x))dy\]
and the difference $\left|c\cdot\vol(W(r)) - \vol(W'_r) \right|$ is
\begin{equation*}\label{eq:209}
 \left|c\cdot \int_{y\in W(r)} 1dx - \int_{y\in W(r)} \sum_{\pi(x)=y}\det(J_y(f_x)) dx \right|,
\end{equation*}
which is bounded by $\max_{x\in W(r)} |\sum_{\pi(x)=y}\det(J_y(f_x))-1|\cdot\vol(W(r))$, 
that can be taken to be as small as desired if $\|\tau\circ(Df)\|$ is small enough for any section $f$ of $\pi_r$ and any point in its domain. But this quantity can be made arbitrarily small by taking $\epsilon'$ small enough.
\end{proof}

\begin{lemma}\label{lemma:gs1} If $W\in \psi(U)$ and $(K,\epsilon)^\gs$ is a neighbourhood of $W$ in $\Psi(U)$, there exists an open subset $D\subset \bR_+\times \bR_+$ and a neighbourhood $(K',\epsilon')^\ls$ of $W$ in $\Pr(U)$ such that if $W'\in (\graphvol)^{-1}(D^-)\cap (K',\epsilon')^\ls$, then the covering map $q\colon Q\to W\cap K$ has a single sheet, and therefore $W'\in (K,\epsilon)^\gs$.
\end{lemma}
\begin{proof}
If $W=\emptyset$, observe that $(K,\epsilon)^\ls = (K,\epsilon)^\gs$. Otherwise, let $x\in K\cap W$ not the closest point to the origin, let $r=\|x\|$ and let $v = \vol(W)(r^+)>0$. Then $(r,v)\in \graphvol(W))$, hence $(r,v)\notin \graph\circ(c\cdot \nu(W))$ because $v>0$. Therefore, we can find a small open neighbourhood $A$ of $(r,v)$ and an open neighbourhood $B$ of $\graph\circ(2\cdot vol(W))$ such that $(s,\infty)\times \bR_+\subset B$ and $A\cap B = \emptyset$. Observe that $\frac{1}{2} B$ is a neighbourhood of $\graphvol(W))$. 

Let $(K',\epsilon')$ be as in Lemma \ref{lemma:ls} for the neighbourhood $\frac{1}{2} B$ of $W$, and define $D = A\cap \frac{1}{2}B$. If $W'\in (K',\epsilon')^\ls$, and its corresponding $q\colon Q\to W\cap K$ has $c$ sheets, then $\graphvol(W')\subset \frac{c}{2}B$, which is disjoint from $\frac{1}{2}B$ unless $c=1$. Therefore, if $W'\in (K',\epsilon')^\ls\cap \graphvol^{-1}(D^-)$ then $q$ is single-sheeted and $W'\in (K',\epsilon')^{\gs}$
\end{proof}

\begin{lemma}\label{lemma:gs2} If $W\in \psi(U)$, and $D^+$ is a subbasic open neighbourhood of $\graphvol(W)$ in $F_0(\bR_+)$ as in Lemma \ref{lemma:subbasic}, then there exists a neighbourhood $(K,\epsilon)^\gs$ of $W$ in $\Psi(U)$ such that $(K,\epsilon)^\gs\subset(\graphvol)^{-1}(D^{+})$.
\end{lemma}
\begin{proof}
Let $D$ be an open subset of $\bR_+\times \bR_+$ containing $(s,\infty)\times \bR_+$, that is also a neighbourhood of $\graphvol(W)$, i.e., $\graphvol(W)\in D^+$. First we take a neighbourhood $(K,\epsilon)^\ls$ of $W$ as in Lemma \ref{lemma:ls} for the neighbourhood $D$ of $\graphvol(W)$. Then, if $W'\in (K,\epsilon)^\gs$, then $W'\in (K,\epsilon)^\ls$ and its corresponding covering map $q\colon Q\to W\cap K$ has a single sheet, so in the mentioned lemma we have $c=1$, so 
\[\graphvol(W')\subset D,\]
hence $\graphvol(W')\in D^+$.
\end{proof}

\bibliographystyle{amsalpha}
\bibliography{biblio-article}

\end{document}